\documentclass[11pt]{article}
\usepackage{amssymb,amsthm,amsmath}

\usepackage{palatino}
\usepackage{rotating}
\usepackage{hyperref}
\usepackage{mdwlist}
\newtheorem{thm}{Theorem}[section]
\newtheorem*{thm*}{Theorem}
\newtheorem{lemma}[thm]{Lemma}
\newtheorem*{lemma*}{Lemma}

\newtheorem*{claim*}{Claim}
\newtheorem{cor}[thm]{Corollary}

\newtheorem{prob}{Problem} 
\newtheorem*{prob*}{Problem}

\theoremstyle{definition}

\newtheorem*{define*}{Definition}
\newtheorem*{note*}{Notation}

\renewenvironment{proof}[1][]{\begin{trivlist}
\item[\hspace{\labelsep}{\bf\noindent Proof#1.\/}] }{\qed\end{trivlist}}

\newcommand{\remove}[1]{}

\title{Large monochromatic triple stars in edge colourings}
\date{21 April, 2013}

\author{Shoham Letzter
\thanks{Department of Pure Mathematics and Mathematical Statistics, Centre for Mathematical Sciences, Wilberforce
Road, Cambridge, CB3 0WB, UK. Email: s.letzter@dpmms.cam.ac.uk }
}
\begin{document}

\maketitle

\begin{abstract}
\setlength{\parindent}{0in} 
\setlength{\parskip}{.08in} 
\noindent
Following problems posed by Gy\'arf\'as \cite{gyarfas_survey}, we show that for every $r$-edge-colouring of $K_n$ there is a monochromatic triple star of order at least $n/(r-1)$, improving Ruszink\'o's result \cite{diam5}.

An edge colouring of a graph is called a local $r$-colouring if every vertex spans edges of at most $r$ distinct colours. We prove the existence of a monochromatic triple star with at least $rn/(r^2-r+1)$ vertices in every local $r$-colouring of $K_n$. 
\setlength{\parskip}{.1in} 
\end{abstract}

\section{Introduction}
A very simple observation, remarked by Erd\H{o}s and Rado, is that when the edges of $K_n$ are $2$-coloured there exists a monochromatic spanning component.
One can generalize this  and look for large monochromatic components satisfying certain conditions. For example, it is an easy exercise to show that every $2$-colouring of $K_n$ has a spanning component of diameter at most $3$ (see \cite{manuscript}, \cite{mubayi}).
As a further generalization, one can consider edge colourings with more than two colours.
Gy\'arf\'as \cite{gyarfas} extended the above observation by showing that every $r$-colouring of $K_n$ has a monochromatic component with at least $n/(r-1)$ vertices. This is tight when there exists an affine space of order $r-1$ and $(r-1)^2$ divides $n$. F\"uredi \cite{furedi} improved this bound in the case when there exists no affine space of order $r-1$, showing that for such $r$ every $r$-colouring of $K_n$ has a monochromatic component with at least $n/(r-1-(r-1)^{-1})$ vertices.

A \emph{double star} is a tree obtained by joining the centres of two stars by an edge.
Gy\'arf\'as \cite{gyarfas_survey} proposed the following problem.
\begin{prob}\label{prob_1}
 Is it true that for $r\ge 3$  every $r$-colouring of $K_n$ contains a monochromatic double star of size at least $n/(r-1)$?
\end{prob}
For $r=2$ the answer to this question is negative. It is shown in \cite{double_star} (and  implicitly in \cite{domination}) that when $K_n$ is $2$-coloured there is a monochromatic double star of size at least $3n/4$. This can be shown to be asymptotically tight using random graphs.
The best known result so far for $r\ge 3$ was obtained by Gy\'arf\'as and S\'ark{\"o}zy \cite{double_star}. They showed that when the edges of $K_n$ are $r$-coloured there is a monochromatic double star of size at least $\frac{n(r+1)+r-1}{r^2}$.

A weaker version of the above problem is as follows. 
\begin{prob}\label{prob_2}
Is there a constant $d$ for which in every $r$-colouring of $K_n$ there exists a monochromatic component of diameter at most $d$ and size at least $n/(r-1)$?
\end{prob}
Note that an affirmative answer to the first problem implies an affirmative answer to this one with $d=3$, which would be best possible (see 	\cite{fowler}).
Ruszink\'o \cite{diam5} solved the last problem with $d=5$, showing that for every $r$-colouring of $K_n$ there is a monochromatic component of diameter at most $5$ with at least $n/(r-1)$ vertices.

The first main result of this short note proves a weaker version of the first problem. A \emph{triple star} is a tree obtained by joining the centres of three stars by a path of length $2$.
\begin{thm}\label{thm_triple_star}
Let $G=K_n$ be $r$-edge-coloured with $r\ge 3$. Then $G$ contains a monochromatic triple star with at least $n/(r-1)$ vertices.
\end{thm}
Note that this is sharp in certain cases, namely whenever $n/(r-1)$ is a sharp lower bound for general monochromatic components in $r$-colourings of $K_n$.
The claim in the theorem does not hold for $r=2$. In \cite{domination} it is implicitly shown that every $2$-coloured $K_n$ contains a monochromatic triple star of size $7n/8$. Furthermore, this is shown to be asymptotically tight using random colourings. 
As an immediate corollary of theorem \ref{thm_triple_star} we answer problem \ref{prob_2} with $d=4$, improving Ruszink\'o's result. 
\begin{cor}
Let $r\ge 3$.
In every $r$-colouring of $K_n$ there is a monochromatic subgraph of diameter at most $4$ on at least $n/(r-1)$ vertices.
\end{cor}

A \emph{local $r$-colouring} is an edge colouring in which for every vertex the edges incident to it have at most $r$ distinct colours.
In \cite{local_colouring} it is shown that in every local $r$-colouring of $K_n$ there is a monochromatic component with at least $\frac{rn}{r^2-r+1}$ vertices. This is sharp when  there exists a projective plane of order $r-1$ and $r^2-r+1$ divides $n$.
In \cite{double_star} it is shown that in local $r$-colourings of $K_n$ there is a monochromatic double star of size at least $\frac{(r+1)n+r-1}{r^2+1}$. Moreover, it is shown that for local $2$-colouring of $K_n$ there exists a monochromatic double star of size at least $2n/3$, and, as mentioned above, this is a sharp lower bound for the size of general monochromatic connected components.
Our second main result shows that the above  lower bounds for monochromatic components can be achieved also for components which are triple stars. Namely,
\begin{thm} \label{thm_local_triple_stars}
Let $G=K_n$ be $r$-locally-coloured with $r\ge 3$. Then $G$ contains a monochromatic triple star with at least $\frac{rn}{r^2-r+1}$ vertices.
\end{thm}
 
As before, the following corollary is immediate.
\begin{cor}
Let $r\ge 3$. In every $r$-local-colouring of $K_n$ there exists a monochromatic component of diameter at most $4$ with at least $\frac{rn}{r^2-r+1}$ vertices.
\end{cor} 
We prove theorem \ref{thm_triple_star} in section \ref{sec_colouring}, and theorem \ref{thm_local_triple_stars} in section \ref{sec_local_colourings}. In the last section \ref{sec_conclusion} we finish with some concluding remarks and open problems.

\section{Triple stars in  edge colourings} \label{sec_colouring}

\begin{proof}[ of theorem \ref{thm_triple_star}]
We assume to the contrary of the statement in the theorem that $G$ contains no monochromatic triple star of the given size.
Let $G_1$ be a subgraph of $G$ which is a monochromatic double star of maximal order and let $U$ be its vertex set. Denote the colour of the edges of $G_1$  by $r$.
By our assumption $|U|<  n/(r-1)$. Let $a>0$ satisfy $|U|= n/(r-1)-a$ (note that $a$ may not be an integer).

Consider the bipartite graph $G_2$ with bipartition $U\cup (V(G)\setminus U)$ and edge set $E$, containing the edges between $U$ and $V(G)\setminus U$  not coloured by $r$ in $G$.
Note that for every vertex $u\in U$ less than $a$ edges between $u$ and $V(G)\setminus U$ have colour $r$, as otherwise there would be an $r$-coloured triple star with at least $ n/(r-1)$ vertices, contradicting our assumption.
Therefore
\begin{equation}\label{eqn_num_of_edges}
|E|>  |U|(n-|U|)-a|U|.
\end{equation}

We  use the following lemma which is due to Mubayi \cite{mubayi} and Liu, Morris and Prince \cite{highly_connected}. We present the proof here for the sake of completeness. 

\begin{lemma}\label{lem_bipartite_double_star}
Let $G=(V, E)$ be a bipartite graph with bipartition $V=A\cup B$. Then $G$ contains a double star with at least $(\frac{1}{|A|}+\frac{1}{|B|})|E|$ vertices.
\end{lemma}

\begin{proof}
For a vertex $v\in V$, let $d(v)$ denote the degree of $v$ in $G$ and for an edge $e=(a,b)\in E$, let $c(e)=d(a)+d(b)$.
By the Cauchy-Schwartz inequality, 
\begin{align*}
&\sum_{e\in E}c(e)=\sum_{a\in A}d(a)^2+\sum_{b\in B}d(b)^2\ge
\frac{1}{|A|}(\sum_{a\in A}d(a))^2+\frac{1}{|B|}(\sum_{b\in B}d(b))^2=(\frac{1}{|A|}+\frac{1}{|B|})|E|^2.
\end{align*}
Therefore, there is an edge $e\in E$ with $c(e)\ge(\frac{1}{|A|}+\frac{1}{|B|})|E|$, i.e.~$G$ contains a double star of the required size.
\end{proof}
By considering the edges with the majority colour, the lemma implies that $G_2$ has a monochromatic double star $G_3$ with at least $(\frac{1}{|U|}+\frac{1}{n-|U|})\frac{|E|}{r-1}$ vertices.
Using inequality \ref{eqn_num_of_edges} for the size of $E$ and the expression for the size of $U$, $G_3$ has at least the following number of vertices.
\begin{align*}
&\frac{1}{r-1}\cdot\frac{n}{|U|(n-|U|)}\cdot\,(\,|U|(n-|U|)-a|U|\,)=\\
&\frac{n}{r-1}-a\frac{n}{r-1}(\frac{1}{ \frac{r-2}{r-1}n +a})>\\
&\frac{n}{r-1}-\frac{a}{r-2} \ge \frac{n}{r-1}-a=|U| 
\end{align*}
Note that we use here the fact that $r\ge 3$.
This implies that $G_3$ has more than $|U|$ vertices,  contradicting the choice of $U$ as the vertex set of the largest monochromatic double star of $G$.
We have thus reached a contradiction to the initial assumption, i.e.~$G$ contains a triple star of the required size.
\end{proof}

\section{Triple stars in local edge colourings}\label{sec_local_colourings}
\begin{proof}[ of theorem \ref{thm_local_triple_stars}]
As in the proof of theorem \ref{thm_triple_star}, we take $U$ to be the vertex set of the largest monochromatic double star, and assume it has $\frac{rn}{r^2-r+1}-a$ vertices, where $a>0$.
We define the bipartite graph $G_2$ as before, and obtain the same inequality \ref{eqn_num_of_edges} for $|E|$.
The following lemma generalizes lemma \ref{lem_bipartite_double_star} from the previous section. A weaker form of this lemma appears in \cite{local_colouring}. 
\begin{lemma}
Let  a bipartite graph $G=(V, E)$ with bipartition $V=A\cup B$ be edge coloured. 
Let $r,t$ be such that every vertex $x\in A$ is incident to edges of at most $r$ distinct colours, and every vertex $y\in B$ is incident to edges of at most $t$ colours.
Then $G$ contains a monochromatic double star with at least $(\frac{1}{|A|r}+\frac{1}{|B|t})|E|$ vertices.
\end{lemma}
\begin{proof}
For a vertex $v\in A\cup B$ denote by $I(v)$ the number of colours used in the set of edges in $G$ incident with $v$. For a colour $k$ denote by $d_k(v)$ the number of $k$-coloured-edges containing $v$.
For en edge $e=(a,b)$ in $G$ of colour $k$ let $c(e)=d_k(a)+d_k(b)$.
Then by the Cauchy-Schwartz inequality, using the properties of the colouring,
\begin{align*}
&\sum_{e\in E}c(e)=\sum_{a\in A}\sum_{k\in I(a)}d_k(a)^2+\sum_{b\in B}\sum_{k\in I(b)}d_k(b)^2\ge\\
&\frac{1}{|A|r}(\sum_{a\in A}\sum_{k\in I(a)}d_k(a))^2+\frac{1}{|B|t}(\sum_{b\in B}\sum_{k\in I(b)}d_k(b))^2=(\frac{1}{|A|r}+\frac{1}{|B|t})|E|^2.
\end{align*}
And the claim follows.
\end{proof}
Note that every vertex in $V(G)\setminus U$ spans edges of at most $r$ colours in $G_2$ and every vertex in $U$ spans edges with at most $r-1$ colours in $G_2$, using the fact that $G$ is locally $r$-coloured, and the definition of $G_2$.
Thus $G$ contains a monochromatic double star with at least the following number of vertices.
\begin{align*}
&(\frac{1}{|U|(r-1)}+\frac{1}{(n-|U|)r})|E|\,>\,
(\frac{1}{|U|(r-1)}+\frac{1}{(n-|U|)r})(|U|(n-|U|)-a|U|)=\\
&\frac{n-|U|}{r-1}+\frac{|U|}{r}-\frac{a}{r-1}-a\frac{|U|}{(n-|U|)r}=\\
&\frac{(r-1)n}{r^2-r+1}+\frac{a}{r-1}+\frac{n}{r^2-r+1}-\frac{a}{r}-\frac{a}{r-1}-a\frac{\frac{rn}{r^2-r+1}-a}{\frac{(r-1)^2rn}{r^2-r+1}+a}\ge\\
&\frac{rn}{r^2-r+1}-\frac{a}{r}-\frac{a}{(r-1)^2}\,>\,
\frac{rn}{r^2-r+1}-a=|U|.
\end{align*}
As in theorem \ref{thm_triple_star}, we reached a contradiction to the choice of $U$, thus we have a monochromatic triple star of the required size.
\end{proof}

\section{Concluding Remarks}\label{sec_conclusion}
Problem \ref{prob_1} which is the original question posed by Gy\'arf\'as, remains open. Is it true that for $r\ge3$ every $r$-colouring of $K_n$ contains a monochromatic triple star with at least $n/(r-1)$ vertices? It may also be interesting to consider the weaker version of this question, taking $d=3$ in problem \ref{prob_2}. Does every $r$-colouring of $K_n$ contain a diameter $3$ monochromatic subgraph of size at least $n/(r-1)$?
Finally, it may be interesting to address the same questions in the context of local $r$-colourings (for $r\ge 3$). Namely, is it true that every local $r$-colouring contains a component of diameter at most $3$ with at least $\frac{rn}{r^2-r+1}$ vertices? If so, is there such a component which is a double star?

\section*{Acknowlodgements}
I would like to thank Mikl\'os Ruszink\'o for introducing me to the subject of finding large monochromatic components in edge colourings of $K_n$ and for some useful discussions.

\bibliography{bib}
\bibliographystyle{plain}
\end{document}